\documentclass[10pt]{article}
\usepackage[usenames]{color}
\usepackage{amssymb}
\usepackage{amsfonts,amsmath}
\usepackage{multicol}
\usepackage{graphicx}
\usepackage{amsthm}
\usepackage{verbatim}
\usepackage{float}
\usepackage{forest,adjustbox}
\usepackage{color}
\usepackage{mathrsfs}
\usepackage{amsmath,url}
\usetikzlibrary{arrows,decorations.pathmorphing,decorations.markings,backgrounds,positioning,fit,petri,calc}
\usepackage{tikz}
\usetikzlibrary{arrows.meta}
\usepackage{amsmath, epsfig}

\theoremstyle{plain}
\newtheorem{thm}{Theorem}
\newtheorem{lemma}[thm]{Lemma}
\newtheorem{cor}[thm]{Corollary}
\newtheorem{prop}[thm]{Proposition}

\newtheorem{conj}[thm]{Conjecture}

\newtheorem{observation}{Observation}
\newtheorem*{ques*}{Question}
\usetikzlibrary{graphs}

\title{Results on the Small Quasi-Kernel Conjecture}
\author{
Jiangdong Ai\thanks{Department of Computer Science. Royal Holloway University of London.  {\tt Jiangdong.Ai.2018@live.rhul.ac.uk}.} \and Stefanie Gerke\thanks{Department of Mathematics. Royal Holloway University of London.  {\tt stefanie.gerke@rhul.ac.uk}.} \and Gregory Gutin \thanks{Department of Computer Science. Royal Holloway University of London. {\tt g.gutin@rhul.ac.uk}.} \and Anders Yeo \thanks {Department of Mathematics and Computer Science, University of Southern Denmark. {\tt andersyeo@gmail.com}.} \and Yacong Zhou\thanks{Department of Computer Science. Royal Holloway University of London. {\tt Yacong.Zhou.2021@live.rhul.ac.uk}.} }
\begin{document}
\maketitle

 \begin{abstract}
A {\em quasi-kernel} of a digraph $D$ is an independent set $Q\subseteq V(D)$ such that for every  vertex $v\in V(D)\backslash Q$, there exists a directed path with one or two arcs from $v$ to a vertex $u\in Q$. In 1974, Chv\'{a}tal and Lov\'{a}sz proved that every digraph has a quasi-kernel. In 1976, Erd\H{o}s and S\'zekely conjectured that every sink-free digraph $D=(V(D),A(D))$ has a quasi-kernel of size at most $|V(D)|/2$.
In this paper, we give a new method to show that the conjecture holds for a generalization of anti-claw-free digraphs. For any sink-free one-way split digraph $D$ of order $n$, when $n\geq 3$, we show a stronger result that $D$ has a quasi-kernel of size at most $\frac{n+3}{2} - \sqrt{n}$, and the bound is sharp.
\end{abstract}

\section{Introduction}\label{sec:intro}

Let $D=(V(D),A(D))$ be a digraph. A {\em kernel} of a digraph $D$ is an independent set $K\subseteq V(D)$ such that for every vertex $v\in V(D)\backslash K$, there exists an arc from $v$ to a vertex $u\in K$. A {\em quasi-kernel} of a digraph $D$ is an independent set $Q\subseteq V(D)$ such that for every  vertex $v\in V(D)\backslash Q$, there exists a directed path with one or two arcs from $v$ to a vertex $u\in Q$.  Note that not every digraph has a kernel. For instance, odd cycles do not have a  kernel. However, Chv\'{a}tal and Lov\'{a}sz \cite{CL} proved that every digraph has a quasi-kernel, and since then, quasi-kernels have been studied in a number of papers, see e.g. \cite{C,GKTY,HH,H,JM,KLS,LMRV}.

If $xy\in A(D)$ then we say that $y$ is an {\em out-neighbour} of $x$, and $x$ is an {\em in-neighbour} of $y$. A {\em sink} is a vertex that has no out-neighbour in $D$ and a {\em source} is a vertex that has no in-neighbour in $D$. A digraph $D$ is said to be {\em sink-free} if it has no sink. In 1976, Erd\H{o}s and S\'zekely (see \cite{ES}) made the following conjecture on the size of quasi-kernels in sink-free digraphs.

\begin{conj}\label{conj1}\cite{ES}\label{conj}
Every sink-free digraph $D=(V(D),A(D))$ has a quasi-kernel of size at most $|V(D)|/2$.
\end{conj}

Following \cite{ES}, we will call a quasi-kernel {\em small} if its size is at most $|V(D)|/2$. There are a number of papers with results relevant to Conjecture \ref{conj}, cf. \cite{GKTY,HH,H,KLS}. Recently, there was a renewed interest in this conjecture. Kostochka et al. \cite{KLS} proved that Conjecture \ref{conj} holds for a class of digraphs including digraphs of chromatic number at most $4$ (the chromatic number of a digraph $D$ is the same as that of the underlying undirected graph of $D$), see Section \ref{sec:4color} for more details. Recently, van Hulst \cite{H} showed that Conjecture \ref{conj} holds for all digraphs containing kernels. Heard and Huang \cite{HH} showed that each sink-free digraph $D$ has two disjoint quasi-kernels if $D$ is semicomplete multipartite, quasi-transitive, or locally semicomplete. As a consequence, Conjecture \ref{conj} is true for these three classes of digraphs. Note that not every sink-free digraph has a pair of disjoint quasi-kernels \cite{GKTY}.

In this paper we continue to investigate which classes of digraphs satisfy Conjecture \ref{conj}. Our set-up then allows us to give short proofs for Kostochka et al.'s  and van Hulst's results and we include them even though they follow similar lines.
An induced subgraph $C$ of an undirected graph $G$ is called a {\em claw} if $C$ is isomorphic to $K_{1,3}.$ An undirected graph $G$ is {\em claw-free} if it does not have a claw. Claw-free graphs are widely studied; see for example the  surveys \cite{CS,FFR} and the references therein.
Similarly, an induced subgraph $C$ of a digraph $D$ is called a {\em (directed) claw} ({\em (directed) anti-claw}) if $C$ is isomorphic to $\vec{K}_{1,3}$ ($\vec{K}_{3,1}$, respectively) depicted in Figures \ref{fig:1} and \ref{fig:2}, respectively.
\begin{figure}
	\centering
	\begin{minipage}{.4\textwidth}
		\centering
		\begin{tikzpicture}[scale=1.7]
 \fill (0,0) circle (0.025);
\fill (-1,0) circle (0.025);
\fill (-1,1) circle (0.025);
\fill (-1,-1) circle (0.025);
\draw[thick] (-1,0)--(-0.5,0);
\draw[thick] (-1,1)--(-0.5,0.5);
\draw[thick] (-1,-1)--(-0.5,-0.5);
\draw [thick, arrows = {-Stealth[reversed, reversed]}] (0,0)--(-0.5,0);
\draw [thick, arrows = {-Stealth[reversed, reversed]}] (0,0)--(-0.5,0.5);
\draw [thick, arrows = {-Stealth[reversed, reversed]}] (0,0)--(-0.5,-0.5);
		\end{tikzpicture}
		\caption{$\vec{K}_{1,3}$}\label{fig:1}
	\end{minipage}%
	\begin{minipage}{.4\textwidth}
		\centering
		\begin{tikzpicture}[scale=1.7]
 \fill (0,0) circle (0.025);
\fill (-1,0) circle (0.025);
\fill (-1,1) circle (0.025);
\fill (-1,-1) circle (0.025);
\draw[thick] (-0.5,0)--(0,0);
\draw[thick] (-0.5,0.5)--(0,0);
\draw[thick] (-0.5,-0.5)--(0,0);
\draw [thick, arrows = {-Stealth[reversed, reversed]}] (-1,0)--(-0.5,0);
\draw [thick, arrows = {-Stealth[reversed, reversed]}] (-1,1)--(-0.5,0.5);
\draw [thick, arrows = {-Stealth[reversed, reversed]}] (-1,-1)--(-0.5,-0.5);
		\end{tikzpicture}
		\caption{$\vec{K}_{3,1}$}\label{fig:2}
	\end{minipage}
\end{figure}

We say that $D$ is {\em anti-claw-free} if $D$ contains no anti-claw. In Section \ref{sec:cl}, we prove that every sink-free anti-claw-free digraph has a small quasi-kernel. Note that tournaments are anti-claw-free and so it is easy to construct wide classes of anti-claw-free with large chromatic number. Thus, our result can be viewed as ''orthogonal'' to that by Kostochka et al. \cite{KLS}.

A quasi-kernel $Q$ of a digraph $D$ is {\em good} if for each $u\in Q$, there is an arc from $u$ to a vertex which is an in-neighbour of a vertex in $Q$. In Section \ref{sec:good}, we prove that if a digraph $D$ has a good quasi-kernel, then $D$ has a small quasi-kernel.
Note that every kernel in a sink-free digraph is also a good quasi-kernel and thus our result generalises that of van Hulst \cite{H} and has a shorter proof.
For completeness we also exhibit an infinite set of kernel-free digraphs which have good quasi-kernels.

A {\em split graph} is a graph in which vertices can be partitioned into a clique and an independent set. Split graphs are one of the basic classes of perfect graphs and a superset of the threshold graphs, see \cite{MP}. In 2012, LaMar \cite{LaMar} introduced split digraphs as a generalization of split graphs and characterized split digraphs by their degree sequences using Fulkerson inequalities \cite{Fulkerson}. Here, we consider a special case of split digraphs, namely one-way split digraphs. A digraph $D$ is called a {\em one-way split digraph}, if its vertex set can be partitioned into $X$ and $Y$, such that $X$ induces an independent set and $Y$ induces a semicomplete digraph (a digraph in which there is at least one arc between every pair of vertices) and any arcs between $X$ and $Y$ go from $X$ to $Y$. Note that all vertices in $X$ are sources. In Section \ref{sec:odsg}, we prove that if $D$ is a one-way split digraph of order $n$ with no sinks, then $D$ has a quasi-kernel of size at most $\frac{n+3}{2} - \sqrt{n}$. We also construct, for infinitely many values of $n$, one-way split digraphs of order $n$ such that the minimum size of their quasi-kernels is $\frac{n+3}{2} - \sqrt{n}$.

\paragraph{Additional terminology and notation.}
For any subset $S\subseteq V(D)$, the {\em in-neighbourhood} of $S$ is defined as follows:
$$N^-(S)=\{x\in V(D)\backslash S: \exists y\in S, (x,y)\in A(D)\}.$$ And the {\em closed in-neighbourhood} of $S$ is defined as $N^-[S]=N^-(S)\cup S$.
We use $N^{--}(S)$ to denote the in-neighbourhood of $N^-[S]$, i.e., $N^{--}(S)=N^-(N^-[S])$. In other words, $N^-(S)$ is the set of vertices with distance one to $S$, and $N^{--}(S)$ is the set of vertices with distance exactly two to $S$. Note that every quasi-kernel $Q$ gives rise to  a partition of $V(D)$ into $Q$, $N^{-}(Q)$ and $N^{--}(Q)$.

For any subset $S$ of $V(D)$, we use $D[S]$ to denote the subgraph induced by $S$, and $D-S$ the subgraph induced by $V(D)\backslash S$.
A digraph $D$ is {\em bipartite} if there is a partition $U,W$ of its vertex set such that no arc has both end-points in either $U$ or $W$; $U$ and $W$ are {\em partite sets} of $D.$
Let $D$ be a bipartite digraph with partite sets $U$ and $W$. An arc set $M$ of $D$ is {\em matching from $U$ to $W$} if $M$ is a matching in the underlying undirected graph of $D$
and all arcs of $M$ are directed from $U$ to $W.$

\section{Small Kernels in Anti-claw-free Digraphs}\label{sec:cl}

Our proofs use the following set-up. Let $Q$ be a  quasi-kernel in a sink-free digraph $D$. We want to find a one-to-one correspondence between $V(D)\backslash Q$ and $Q$ which saturates $Q$ to show that $|Q|\leq |V(D)|/2$. To do so we start with a maximum matching $M$ from $N^{-}(Q)$ to $Q$. Let $M_1$
be the subset of $Q$ covered by $M$, $M_2$ the subset of $N^-(Q)$ covered by $M$, and $A= Q\setminus M_1$, see Figure~\ref{fig:3}.
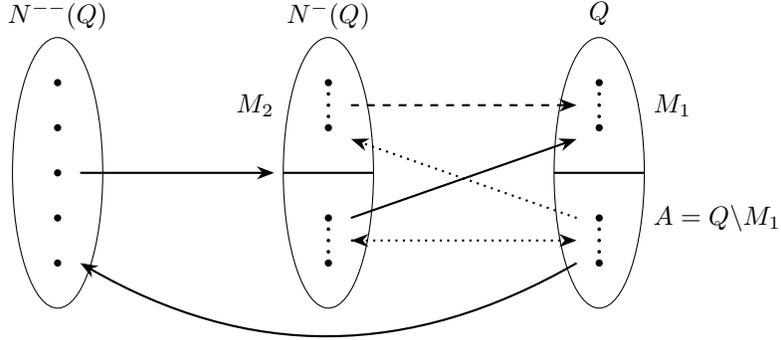
\begin{figure}
	\centering
	\begin{tikzpicture}[scale=0.6]
	\fill (2,4) circle (0.08);
	\fill (2,3) circle (0.08);
	\fill (2,2) circle (0.08);
	\fill (2,1) circle (0.08);
	\fill (2,0) circle (0.08);
	\fill (8,4) circle (0.08);
	\fill (8,3.25) circle (0.04);
	\fill (8,3.5) circle (0.04);
	\fill (8,3.75) circle (0.04);
	\fill (8,3) circle (0.08);
	\fill (8,1) circle (0.08);
	\fill (8,0) circle (0.08);
	\fill (14,4) circle (0.08);
	\fill (14,3.25) circle (0.04);
	\fill (14,3.5) circle (0.04);
	\fill (14,3.75) circle (0.04);
	\fill (14,3) circle (0.08);
	\fill (14,1) circle (0.08);
	\fill (14,0) circle (0.08);
	\fill (14,0.25) circle (0.04);
	\fill (14,0.5) circle (0.04);
	\fill (14,0.75) circle (0.04);
	\fill (8,0.25) circle (0.04);
	\fill (8,0.5) circle (0.04);
	\fill (8,0.75) circle (0.04);
	\draw [black] (2,2) ellipse (1 and 3);
	\draw [black] (8,2) ellipse (1 and 3);
	\draw [black] (14,2) ellipse (1 and 3);
	\node [right] at (15,3.5) {$M_1$};
	\node [right] at (15,1) {$A=Q\backslash M_1$};
	\node [left] at (7,3.5) {$M_2$};
	\node [below] at (2,6) {$N^{--}(Q)$};
	\node [below] at (8,6) {$N^-(Q)$};
	\node [below] at (14,6) {$Q$};
	\draw [thick] (7,2)--(9,2);
	\draw [thick] (13,2)--(15,2);
	\draw [thick, arrows = {-Stealth[reversed, reversed]}] (2.5,2)--(6.8,2);
	\draw [thick, dashed, arrows = {-Stealth[reversed, reversed]}] (8.5,3.5)--(13.5,3.5);
	\draw [thick, dotted, arrows = {-Stealth[reversed, reversed]}] (13.5,1)--(8.5,2.75);
	\draw [thick, arrows = {-Stealth[reversed, reversed]}] (8.5,1)--(13.5,2.75);
	\draw [thick, dotted, arrows = {-Stealth[reversed, reversed]}] (8.5,0.5)--(13.5,0.5);
	\draw [thick, dotted, arrows = {-Stealth[reversed, reversed]}] (8.51,0.5)--(8.5,0.5);
	\draw [thick, arrows = {-Stealth[reversed, reversed]}] (13.5,0) to [bend left] (2.5,0);
	
	\end{tikzpicture}
	\caption{A minimal quasi-kernel of $D$. Both $Q$ and $N^-(Q)$ have been partitioned into two parts. The dotted arcs indicate there are no arcs from one part to another. Ordinary arcs indicate that each vertex in the first part has an out-neighbour in the second part, while the dashed arc indicates the matching $M$. }\label{fig:3}
\end{figure}
 Note that the maximality of $M$ implies that $N^-(Q)= N^-(M_1)$. Thus, we have the following observation.
\begin{observation}\label{obs0} $M_1$ is a quasi-kernel of $D[V(D)\setminus A]$.
\end{observation}
In other words,  every vertex outside $A$ can reach $M_1$ by a directed path with at most two arcs. {Moreover, no vertex has an arc into $A$ that does not have an arc into $M_1$ which leads us to the following observation.}
\begin{observation}
If $Q$ is a quasi-kernel in a sink-free digraph, then deleting any subset of vertices of A yields a sink-free digraph.
\end{observation}

If there is an arc from a vertex $v\in A$  to $N^{-}(Q)$ then we can remove $v$ from $Q$ to obtain a smaller quasi-kernel as $v$ is in the second in-neighbourhood of a vertex in $M_1$. Thus, we obtain the following observation.
\begin{observation} \label{obs1}  If $Q$ is a minimal quasi-kernel, then there is no arc from $A$ to $N^{-}(Q)$.
\end{observation}

Now we are ready to confirm the conjecture for anti-claw-free digraphs.

 \begin{thm} \label{thm:claw}
Every sink-free digraph with no anti-claw has a small quasi-kernel.
\end{thm}
\begin{proof}
Assume there is a sink-free and anti-claw-free digraph without a small quasi-kernel. Let $Q$ be the smallest quasi-kernel and consider a maximal matching $M$ from $N^{-}(Q)$ to $Q$. Recall that $N^{--}(Q)$ is the set of vertices that have a path with two arcs into $Q$ but are not in $N^{-}[Q]$. As $Q$ is not small $|N^{-}(Q)|+|N^{--}(Q)| <|Q|$ and in particular
\begin{equation}\label{eq1}
|N^{--}(Q)|<|A|
 \end{equation}  We have seen in Observation~\ref{obs1} that all out-neighbours of vertices in $A$ must be in $N^{--}(Q)$ (since the digraph is sink-free, every vertex in $A$ must have at least one out-neighbour).  $N^{--}(Q)$ is smaller than $A$ and thus there must exist a vertex  $v\in N^{--}(Q)$ that has (at least) two in-neighbours in $A$. Note that $v$ cannot have an out-neighbour in $Q$ as it is not in the first out-neighbourhood of $Q$. Now $v$ must have an in-neighbour in $M_1$ since otherwise we could replace the in-neighbourhood of $v$ in $Q$ by $v$ to obtain a smaller quasi-kernel. Thus, $v$ and its three in-neighbours form an  induced anti-claw.
\end{proof}

Now, let's consider the second in-neighbourhood of $Q$, or more precisely the set $N^{--}(Q)$ of all vertices that have a path with two arcs into $Q$ but are neither in $N^{-}(Q)$ nor in $Q$.
Let $\tilde{Q}$ be a quasi-kernel of $D[N^{--}(Q)]$. Note that there are no arcs from $\tilde{Q}$ to $Q$ (as $\tilde{Q}$ is in the second in-neighbourhood of $Q$). Hence,  if we add $\tilde{Q}$ to $Q$ and remove $N^{-}(\tilde{Q})$, we obtain an independent set $Q'$. In fact, $Q'$ is a quasi-kernel of $D$ as every vertex in $N^{-}(Q)$ either has an out-neighbour in $Q\setminus N^-(\tilde{Q})$ or has an out-neighbour in $N^-(\tilde{Q})$ and so has a path of length two to a vertex in $\tilde{Q}$. This way of constructing a different quasi-kernel from the original one was first found by Jacob and Meyniel \cite{JM}.

\begin{observation} \label{obs:N2}\cite{JM}
Let $Q$ be a quasi-kernel of a digraph $D$, and let $\tilde{Q}$ be a quasi-kernel of $D[N^{--}(Q)]$. Then
$(Q \cup \tilde{Q})\setminus N^-(\tilde{Q})$ is a quasi-kernel of $D$.
\end{observation}

We can now strengthen Theorem~\ref{thm:claw} by excluding $\vec{K}_{4,1}$ and $\vec{K}_{4,1}^+$ instead of an anti-claw, see Figures \ref{fig:41} and \ref{fig:41+}. Note that  $\vec{K}_{4,1}$ and $\vec{K}_{4,1}^+$ both include an anti-claw and thus Theorem \ref{4claw} is indeed stronger than
Theorem~\ref{thm:claw}.

\begin{figure}
	\centering
	\begin{minipage}{.5\textwidth}
		\centering
		\begin{tikzpicture}[scale=1.7]
 \fill (1,0) circle (0.025);
\fill (-1,0.5) circle (0.025);
\fill (-1,1.5) circle (0.025);
\fill (-1,-1.5) circle (0.025);
\fill (-1,-0.5) circle (0.025);
\draw[thick] (-1,0.5)--(1,0);
\draw[thick] (-1,1.5)--(1,0);
\draw[thick] (-1,-0.5)--(1,0);
\draw[thick] (-1,-1.5)--(1,0);
\draw [thick, arrows = {-Stealth[reversed, reversed]}] (-1,0.5)--(0,0.25);
\draw [thick, arrows = {-Stealth[reversed, reversed]}] (-1,1.5)--(0,0.75);
\draw [thick, arrows = {-Stealth[reversed, reversed]}] (-1,-1.5)--(0,-0.75);
\draw [thick, arrows = {-Stealth[reversed, reversed]}] (-1,-0.5)--(0,-0.25);
		\end{tikzpicture}
		\caption{$\vec{K}_{4,1}$}\label{fig:41}
	\end{minipage}%
	\begin{minipage}{.5\textwidth}
		\centering
		\begin{tikzpicture}[scale=1.7]
 \fill (1,0) circle (0.025);
\fill (-1,0.5) circle (0.025);
\fill (-1,1.5) circle (0.025);
\fill (-1,-1.5) circle (0.025);
\fill (-1,-0.5) circle (0.025);
\draw[thick] (-1,0.5)--(1,0);
\draw[thick] (-1,1.5)--(1,0);
\draw[thick] (-1,-0.5)--(1,0);
\draw[thick] (-1,-1.5)--(1,0);
\draw [thick, arrows = {-Stealth[reversed, reversed]}] (-1,0.5)--(0,0.25);
\draw [thick, arrows = {-Stealth[reversed, reversed]}] (-1,1.5)--(0,0.75);
\draw [thick, arrows = {-Stealth[reversed, reversed]}] (-1,-1.5)--(0,-0.75);
\draw [thick, arrows = {-Stealth[reversed, reversed]}] (-1,-0.5)--(0,-0.25);
\draw [thick, arrows = {-Stealth[reversed, reversed]}] (-1,-0.5)--(-1,-1.5);
		\end{tikzpicture}
		\caption{$\vec{K}_{4,1}^+$}\label{fig:41+}
	\end{minipage}
\end{figure}

\begin{thm}  \label{4claw}Every sink-free digraph with no induced $\vec{K}_{4,1}$ and no induced $\vec{K}_{4,1}^+$ has a small quasi-kernel.
\end{thm}

\begin{proof}
Let $D$ be a sink-free digraph with no induced $\vec{K}_{4,1}$ and no induced $\vec{K}_{4,1}^+$ and let $Q$ be a minimum quasi-kernel of $D$. Assume for a contradiction that  $Q$ is not small.  Let $\tilde{Q}$ be a maximal quasi-kernel of $D[N^{--}(Q)]$. By Observation \ref{obs:N2}, $Q'=(Q \cup \tilde{Q})\setminus N^-(\tilde{Q})$ is a quasi-kernel of $D$. Moreover, we can remove any vertex in $A=Q\setminus M_1$ from $Q'$ that has an arc to $N^-(\tilde{Q})$ to obtain a smaller quasi-kernel $Q''$ as these vertices have paths of length $2$ into $\tilde{Q}$. Let $A'$ be the set of vertices in $A$  that have an out-neighbour in $\tilde{Q}$ or $N^{-}(\tilde{Q})\cap N^{--}(Q)$. By the minimality of $Q$, we have  $|A'| \leq |\tilde{Q}|$ since
\[ |Q|\leq  |Q''| \leq |Q|+|\tilde{Q}|- |A'|.\]
Thus (using \eqref{eq1})

\begin{align*}
|A\setminus A'| &= |A|-|A'| > |N^{--}(Q)|-|A'|\\& = |\tilde{Q}|+|N^{-}(\tilde{Q})\cap N^{--}(Q)|+|N^{--}(\tilde{Q})\cap N^{--}(Q)| - |A'|\\
& \geq |N^{-}(\tilde{Q})\cap N^{--}(Q)| + |N^{--}(\tilde{Q})\cap N^{--}(Q)|
\end{align*}

As

$|N^{-}(\tilde{Q})\cap N^{--}(Q) | \geq 1$ we have

\begin{equation} \label{eq2}
|A\setminus A'| \geq |N^{--}(\tilde{Q})\cap N^{--}(Q)| +2
\end{equation}

Therefore, there is a vertex   
$v \in N^{--}(\tilde{Q})\cap N^{--}(Q)$  with at least two  in-neighbours in $A\setminus A'$. Moreover $v$ must have an in-neighbour in $M_1$ (otherwise we would start with $Q$ and swap $v$ and its neighbours in $A$ to obtain a quasi-kernel smaller than $Q$) and also an in-neighbour in $\tilde{Q}$ as $\tilde{Q}$ is a maximal quasi-kernel. Clearly the in-neighbours in $A\setminus A'$ and in $M_1$ are independent as they are part of the quasi-kernel $Q$ and there can be no arc from $v$ to any of these in-neighbours as it is in $N^{--}(Q)\cap N^{--}(\tilde{Q})$ and therefore the graph must contain a $\vec{K}_{4,1}$ or a $\vec{K}_{4,1}^+$, a contradiction.
\end{proof}

A digraph in which all vertices have in-degree at most $3$ cannot contain an induced  $\vec{K}_{4,1}$ or $\vec{K}_{4,1}^+$ and therefore has a small quasi-kernel. Then we may conjecture that every sink-free digraph with in-degree at most $4$ has a small quasi-kernel. The following theorem will support this conjecture.

\begin{thm}
Let $D$ be a digraph with no small quasi-kernel. Then $D$ contains a vertex with at least 5 in-neighbours (at least $3$ of which form an independent set)  or two independent arcs such that each vertex of the two arcs has at least $4$ in-neighbours, $3$ of which form an independent set.
\end{thm}
\begin{proof}
We follow the proof of Theorem~\ref{4claw} but choose $\tilde{Q}$ with more care. Consider the set $N_2$ of all vertices in $N^{--}(Q)$ that have at least two in-neighbours in $A$ (and therefore an in-neighbour in $M_1$ by the minimality of $Q$ as before). Let $N_2'$ be a maximum independent set of $N_2$. We choose $\tilde{Q}$ as a maximal quasi-kernel of $D[N^{--}(Q)]$ with as many vertices of $N_2$ in $\tilde{Q}$ or $N^{-}(\tilde{Q})\cap N^{--}(Q)$  as possible and such that all vertices in $N_2'$ are either in $\tilde{Q}$ or in $N^{-}(\tilde{Q})\cap N^{--}(Q)$. The fact that such a $\tilde{Q}$ exists is well-known and follows from the proof in
\cite{CL}: One removes $N'_2$ and all its in-neighbours in $N^{--}(Q)$ and finds (inductively) a quasi-kernel $\tilde{Q}$ of the  digraph induced by the remaining vertices in $N^{--}(Q)$. Now, one can add all the vertices of $N'_2$ that have no (out-)neighbours in $\tilde{Q}$ to $\tilde{Q}$ to form a quasi-kernel of $D[N^{--}(Q)]$ with the required properties.

By \eqref{eq2}, we have two cases.

\vspace{1mm}

\noindent{\bf Case 1:} There is a vertex $v \in  N^{--}(\tilde{Q})$ with $3$ neighbours in $A\setminus A'$.  Then we can argue as in the proof of Theorem~\ref{4claw} that  $v$ has $5$ in-neighbours ($4$ of which are in $Q$ and therefore independent).

\vspace{1mm}

\noindent{\bf Case 2:} There is no vertex $v \in  N^{--}(\tilde{Q})$ with $3$ in-neighbours in $A\setminus A'$. Then there must be two vertices $u$ and $w$ from $N^{--}(\tilde{Q})$, each of which must have two in-neighbours in $A\setminus A'$ and  an in-neighbour in $M_1$ and in $\tilde{Q}$.  If there is an arc say from $u$ to $w$, then $w$ has $5$ in-neighbours (and $3$ of them in $Q$ are independent).
If there is no arc between $u$ and  $w$ then there are two distinct neighbours   $x$ and $y$ in $N_2'$ as $N_2'$ is a maximum independent set of $N_2$ so the two independent vertices $u$ and $w$ cannot have a single vertex in $N_2'$ as their joint neighbourhood in $N_2$.

We assume that $x$ is the neighbour of $u$ (and $y$ is the neighbour of $w$). If there is an arc from $u$ to $x$, then $x$ and $u$ have in-degree at least $4$ and at least $3$ in-neighbours of each vertex are in $Q$ and therefore independent.  If there is an arc from $x$ to $u$ and $x\in N^{-}(\tilde{Q})$, then $u$ has in-degree at least $5$ as $x$ has at least $3$ in-neighbours in $Q$, at least one in-neighbour in $\tilde{Q}$ and at least one in-neighbour, namely $x$, in $N^{-}(\tilde{Q})$.
If $x\in \tilde{Q}$ then we consider three cases. Firstly, if there is another in-neighbour of $u$ in $\tilde{Q}$ then $u$ has at least $5$ in-neighbours ($4$ of which are independent, namely the ones in $A\setminus A'$ and $\tilde{Q}$). Secondly, if $x$ has another in-neighbour then both $x$ and $u$ have $4$ in-neighbours ($3$ of which are in $Q$ and therefore independent). Finally, if neither $x$ nor $u$ have any in-neighbours then we can replace $x$ by $u$ in $\tilde{Q}$ to obtain a quasi-kernel which has one more element of $N_2$ in it or its first in-neighbourhood, a contradiction to the choice of $\tilde{Q}$.
\end{proof}

\section{Alternative proof for a generalization of $4$-colorable digraphs}\label{sec:4color}
In this section we use our set-up to give a short proof of a result  by Kostochka et al. \cite{KLS}. We start with the following lemma.

\begin{lemma}\label{lem2}
	Let $D$ be a sink-free digraph. If $D$ has a quasi-kernel $Q$ such that $D[N^{--}(Q)]$ has a kernel, then $D$ has a small quasi-kernel.
\end{lemma}
\begin{proof}
Let $Q$ be a  quasi-kernel such that  $D[N^{--}(Q)]$ has a kernel, and let $M$ be a maximal matching from $N^-(Q)$ to $Q$.
We may assume that every vertex in $A=Q\backslash M_1$  has an out-neighbour in  $N^{--}(Q)$ as otherwise we can remove $v$  from $Q$ as  $v$ is an isolated vertex in $N^{--}(Q\backslash \{v\})=N^{--}(Q)\cup\{v\}$, and therefore $Q\setminus \{v\}$ is a quasi-kernel such that $N^{--}(Q\backslash \{v\})$ has a kernel.
	
Let $K$ be a kernel of $D[N^{--}(Q)]$. If $|A|\leq |N^{--}(Q)|$, then $|Q|=|M_1|+|A|\leq |M_2|+|N^{--}(Q)|$, which implies $|Q|$ is a small quasi-kernel. If $|A|> |N^{--}(Q)|$, then one can observe that $(M_1\cup K)\backslash N^-(K)$ is a quasi-kernel of $D$ whose size is at most $|M_1|+|K|\leq |M_1|+|N^{--}(Q)|<|M_2|+|A|$, and therefore is a small quasi-kernel.
\end{proof}

 A digraph is {\em kernel-perfect} if every induced subdigraph of it has a kernel. Note that every digraph with no odd cycle is kernel-perfect \cite{R}. In particular, bipartite digraphs are kernel-perfect. In \cite{KLS} the authors consider a slightly more general conjecture, namely that all digraphs with a set of sinks\footnote {In \cite{KLS}, they defined the quasi-kernel in an opposite direction. Namely, they defined it as an independent set $Q$ such that for any $v\notin Q$ there exists a directed path within two arcs from a vertex $u\in Q$ to $v$. So, in their case, they considered a set of sources.} $S$ have a quasi-kernel of size at most $(n+|S| - |N^{-}(S)|)/2$. They show that this conjecture is equivalent to the original conjecture in the sense that the smallest counterexample to the new conjecture does not contain any sink.  The following theorem implies that both conjectures hold for 4-colourable digraphs.

\begin{thm}\cite{KLS}  \label{thm6} Let $D$ be an $n$-vertex digraph and $S$ be the set of sinks of $D$. Suppose that $V (D) \setminus N^{-}[S]$  has a partition $V_1 \cup  V_2$ such that $D[V_i]$ is kernel-perfect for each $i = 1, 2$. Then D has a quasi-kernel of size at most  $(n+|S|-|N^{-}(S)|)/2$.
\end{thm}

\begin{proof}

Let $(V_1,V_2)$ be the partition such that $|V_2|$ is as small as possible and each part is kernel-perfect. Following \cite{KLS} we first show that every vertex in $V_2$ has an out-neighbour in $V_1$. Assume for a contradiction that $v\in V_2$ has no out-neighbour in $V_1$. We will show that  $D[V_1\cup \{v\}]$ is also kernel-perfect.  To do so, let $F$ be a subset of $V_1\cup \{v\}$. If $v\notin F$ then $D[F]$ has a kernel since $D[V_1]$ is kernel-perfect. If $v\in F$,  let $K_0$ be a kernel in $D[F\backslash N^{-}(v)]$ then $K_0\cup \{v\}$ is a kernel of $D[F]$.

Since every vertex in $V_2$ has an out-neighbour in $V_1$,  any kernel of $D[V_1]$ is a quasi-kernel of $D\setminus N^{-}[S]$.
Following \cite{KLS} we remove $S$ and $N^{-}(S)$ from $D$ and consider a kernel $K$  of $D[V_1]$. As before, let $M$ be a maximal matching from $N^{-}(K)$ to $K$, let $M_1$ be the subset of $K$ covered by $M$ and let  $A=K\setminus M_1$. We can remove the set $A' \subset A$  from $K$ of all vertices that do not have an out-neighbour into $V_1 \cup V_2$ as they are in the second neighbourhood of $S$ and the resulting subset of $K$ is still a quasi-kernel of the remaining vertices.
Now arguing as in  the proof of Lemma \ref{lem2} we obtain a small quasi-kernel of $(V_1\cup V_2)\setminus A' $ and together with $S$, this forms a quasi-kernel of $D$ which is sufficiently small for the result to hold.
\end{proof}

\section{Good Quasi-Kernels}\label{sec:good}

\begin{thm}\label{T1}
Every sink-free digraph $D=(V,A)$ with a good quasi-kernel has a small quasi-kernel.
\end{thm}
\begin{proof}
Let $Q$ be a good quasi-kernel of minimum size in $D$. Let $M$, $M_1$ and $M_2$ have the same definitions of those in Section \ref{sec:cl}. Recall that $N^-(M_1)=N^-(Q)$. In addition, $Q\backslash M_1\subseteq N^{--}(M_1)$ by the definition of good quasi-kernels. Thus, $M_1$ is also a good quasi-kernel of $D$, which contradicts to the minimality of $Q$ if $|Q\backslash M_1|>0$. Therefore, we have $|Q|=|M_1|=|M_2|\leq |N^-(Q)|$, which implies that $Q$ is a small quasi-kernel.
\end{proof}

Now we can prove the following result whose Part 2 was recently shown by van Hulst \cite{H} using a longer inductive proof.

\begin{cor}\label{boundedkernel}
Let $D$ be a sink-free digraph and $K$ a kernel of $D$. Then
\begin{enumerate}
\item $K$ is a good quasi-kernel;
\item $D$ has a quasi-kernel of size at most $|V(D)|/2.$
\end{enumerate}
\end{cor}
\begin{proof}
For any $u\in K$, $N^+(u)\subseteq N^-(K)$. Since $D$ is sink-free, for all $u\in K$, $N^+(u)\ne \emptyset$. Therefore, for any $u\in K$, we have $N^+(u)\cap N^-(K)=N^+(u)\ne \emptyset$. Thus, $K$ is a good quasi-kernel.
Part 2 follows from Theorem \ref{T1} and Part 1.
\end{proof}

Let $T$ be a digraph with vertex set $\{u_1,\dots, u_t\}$ and  for every $i\in [t]$,  let $D_i$ be a digraph with vertex set $\{u_{i,j_i}\colon\,
j_i\in [n_i]\}$.
The \emph{composition} of $T$ and $D_1,\dots, D_t$
is the digraph  $H=T[D_1,\dots , D_t]$ with vertex set $\{u_{i,j_i}\colon\,  i\in [t],  j_i\in [n_i]\}$ and arc set
$$A(H)=\cup^t_{i=1}A(D_i)\cup \{u_{i,j_i}u_{p,q_p}\colon\, u_iu_p\in A(T), j_i\in [n_i], q_p\in [n_p]\}.$$

By Corollary  \ref{boundedkernel}, every sink-free digraph with a kernel has a good quasi-kernel. The following proposition shows that the opposite is not true i.e. the set of sink-free digraphs with a kernel is a proper subset of sink-free digraphs with a good quasi-kernel.

\begin{figure}[H]
\begin{center}
\begin{tikzpicture}[scale=1.7]

\fill (-2,1) circle (0.025);
\fill (0,0.5) circle (0.025);
\fill (2,1) circle (0.025);
\fill (-2,-1) circle (0.025);
\fill (0,-0.5) circle (0.025);
\fill (2,-1) circle (0.025);

\draw[thick] (-2,1)--(0,0.5)--(2,1)--(-2,1);
\draw[thick] (-2,-1)--(0,-0.5)--(2,-1)--(-2,-1);
\draw[thick] (0,0.5)--(0,-0.5);
\draw[thick] (2,1)--(0,-0.5);
\draw[thick] (2,-1)--(0,0.5);

\node[above] at (-2.02,1)  {$1$};
\node[above] at (0.02,0.5)  {$2$};
\node[above] at (2.02,1)  {$3$};
\node[below] at (-2.02,-1)  {$4$};
\node[below] at (0.02,-0.5)  {$5$};
\node[below] at (2.02,-1)  {$6$};
\draw [thick, arrows = {-Stealth[reversed, reversed]}] (2,1)--(0,1);
\draw [thick, arrows = {-Stealth[reversed, reversed]}] (-2,1)--(-1,0.75);
\draw [thick, arrows = {-Stealth[reversed, reversed]}] (0,0.5)--(1,0.75);
\draw [thick, arrows = {-Stealth[reversed, reversed]}] (2,-1)--(0,-1);
\draw [thick, arrows = {-Stealth[reversed, reversed]}] (-2,-1)--(-1,-0.75);
\draw [thick, arrows = {-Stealth[reversed, reversed]}] (0,-0.5)--(1,-0.75);
\draw [thick, arrows = {-Stealth[reversed, reversed]}] (2,1)--(1,0.25);
\draw [thick, arrows = {-Stealth[reversed, reversed]}] (2,-1)--(1,-0.25);
\draw [thick, arrows = {-Stealth[reversed, reversed]}] (0,0.5)--(0,0);
\end{tikzpicture}
\end{center}\caption{$D^*$}\label{fig:D}
\end{figure}

\begin{prop}
There is an infinite number of connected digraphs without a kernel but with a good quasi-kernel.
\end{prop}
\begin{proof}
In the first part of the proof, we show that the digraph $D^*$ depicted in Figure \ref{fig:D} has no kernels, but  contains a good quasi-kernel. In the second part of the proof, we obtain from $D^*$ an infinite family of connected digraphs without a kernel but with a good quasi-kernel.

Note that if an independent set $K$ is not a kernel, then every subset of $K$ is not either. Thus, we only need to show that all maximal independent sets of $D^*$ are not kernels. Since $\{1,2,3\}$ and $\{4,5,6\}$ induce triangles, every maximal independent set of $D^*$ contains at most one vertex from each of these two sets.  Observe that there are six maximal independent sets in $D^*$: $\{1,4\}$, $\{2,4\}$, $\{3,4\}$, $\{1,5\}$, $\{1,6\}$, $\{3,6\}$.
Each of the six sets is not a kernel since it has a vertex in the second in-neighbourhood: 2, 3, 1, 6, 2, 1, respectively. Observe that $\{3,6\}$ is a good quasi-kernel since $\{3,6\}$ is a quasi-kernel ($N^-(\{3,6\})=\{2,5\}$ and
$N^{--}(\{3,6\})=\{1,4\}$)  and $N^+(3)=\{5\}$, $N^+(6)=\{2\}$.

Construct a composition $D'=D^*[H_1,\dots,H_6]$ from $D^*$ by substituting each vertex $i$ by $H_i$ where $H_i$ is arbitrary for $i=1,2,\dots,6$. We claim that each $D'$ is the desired digraph. Suppose that $D'$ has a kernel $K'$. Construct a set $K^*$ in $D^*$ by putting $i$ in $K^*$ if and only if $K'$ intersects $H_i$. Observe that $K^*$ is a kernel in $D^*,$
which is a contradiction as $D^*$ has no kernel. Thus, $D'$ has no kernel.
Let $Q_3$ be a quasi-kernel of $H_3$ and $Q_6$ a quasi-kernel of $H_6$; then $Q_3\cup Q_6$ is a good quasi-kernel of $D'$ since $\{3,6\}$ is a good quasi-kernel of $D^*$.
\end{proof}
Unfortunately, there is an infinite number of digraphs which have no good quasi-kernel. Observe that the directed cycle $C_3$ with three vertices has no good quasi-kernel and that arbitrary $C_3[H_1,H_2,H_3],$ where $H_1,H_2,H_3$ are digraphs with no arcs, has no good quasi-kernel.

\begin{prop}\label{composition1}
Let $T$ be a digraph with $m\ge 2$ vertices and with a good quasi-kernel and $H_1,H_2,\dots,H _m$ arbitrary digraphs. Then
the composition $D=T[H_1,H_2,\dots,H _m]$ has a good quasi-kernel. Thus, $D$ has a small quasi-kernel.
\end{prop}
\begin{proof}
 Let $Q=\{u_{p_1},\dots,u_{p_s}\}$ be a good quasi-kernel of $T$, $Q_i$ a quasi-kernel of $D_i$ for $i=p_1,\dots , p_s$, and $Q'=\bigcup_{i\in \{p_1,\dots,p_s\}} Q_i.$ It is easy to verify that $Q'$ is a quasi-kernel of $D$ using the fact that $Q$ is a quasi-kernel of $T$. It remains to apply Theorem \ref{T1}.
\end{proof}

\section{One-way Split Digraphs}\label{sec:odsg}

Recall that a one-way split digraph consists of an independent set $X$ and a semicomplete digraph on  $Y$ and arcs from $X$ to $Y$. Before we consider one-way split digraphs let us show the following simple but useful result on quasi-kernels in semicomplete digraphs. Recall that a set $S$ in a family $\cal F$ of sets is {\em maximal} if there is no set $S'$ in $\cal F$ such that $S\subseteq S'.$ 


\begin{lemma} \label{cor:T}
Let $T=(V,A)$ be a semicomplete digraph. Then every vertex $v\in V$ with a maximal in-neighbourhood forms a quasi-kernel of $T$.
\end{lemma}
\begin{proof}
Let $v$ have a maximal in-neighbourhood in $T$ and let $u$ be an arbitrary vertex of $T$ distinct from $v.$ Then either $uv\in A$ or  $v\in N^-(u)$ and there is a vertex $w\in N^-(v)\setminus N^-(u)$. Thus $uw,wv\in A$. Hence, $\{v\}$ is a quasi-kernel of $T.$
\end{proof}

We now show the following theorem for one-way split digraphs.

\begin{thm}\label{thm:split}
	Let $D$ be a one-way split digraph of order $n$ with no sinks. Then $D$ has a quasi-kernel of size at most $\frac{n+3}{2} - \sqrt{n}$.
	Furthermore, for infinitely many values of $n$ there exists a one-way split digraph of order $n$, with no sink, such that the minimum size of its quasi-kernels is $\frac{n+3}{2} - \sqrt{n}$.
\end{thm}
\begin{proof}
	Let $D$ be a one-way split digraph of order $n$ with no sink. Let $X$ and $Y$ be a partition of $V(D)$ such that $X$ is independent and $Y$ induces a semicomplete digraph and
	all arcs between $X$ and $Y$ go from $X$ to $Y$.
	Note that all quasi-kernels of $D$ consist of a single vertex from $Y$ and the vertices of $X$ at distance $3$ from this vertex. We will now show that there exists a vertex $y$ such that the number of vertices at distance $3$ is small (and in particular we could choose $y$ to be a vertex with the most vertices of $X$ at distance at most $2$ that has a maximal in-neighbourhood in $Y$.)
	
	 For each $x \in X$ let $v(x)$ be any vertex in $N^+(x)$ and for each $y \in Y$ let
	$R(y) = \{ x\in X\; | \; v(x)=y\}$.  We now define a new digraph $H$ as follows. Let $V(H)=X$ and define the arc-set of $H$ as follows.
	
	\[
	A(H) = \{ x_1x_2 \; | \;  v(x_1)=v(x_2) \mbox{ or } v(x_1) v(x_2) \in A(D) \}
	\]
	
	Note that if $x_1,x_2 \in R(y)$ for some $y$ then $x_1 x_2, x_2x_1 \in A(H)$. Furthermore $H$ is a semicomplete digraph as for any $x_1,x_2 \in X$ we either have
	$v(x_1)=v(x_2)$ or $v(x_1) v(x_2) \in A(D)$ or $v(x_1) v(x_2) \in A(D)$ (as $Y$ induces a semicomplete digraph). This implies the following.
	
	\[
	|A(H)| \geq {|X| \choose 2} + \sum_{y \in Y} {|R(y)| \choose 2}
	\]
	
	This implies that there exists a vertex $x \in X$ such that the following holds.
	
	\[
	d_H^-(x) \geq \frac{1}{|X|} \left( \frac{|X|(|X|-1)}{2} + \sum_{y \in Y} \frac{|R(y)| (|R(y)| -1)}{2} \right)
	\]
	
	Given the size of $X$ and $Y$, by Jensen's inequality the sum  $\sum_{y \in Y} \frac{|R(y)| (|R(y)| -1)}{2}$ is minimized when all $|R(y)|$ are the same size (we allow $|R(y)|$ not to be an integer).
	That is, $|R(y)|=|X|/|Y|$. This implies that the following holds.
	
	\[
	\begin{array}{rcl} \vspace{0.3cm}
	d_H^-(x) & \geq &  \frac{1}{|X|} \left( \frac{|X|(|X|-1)}{2} + \sum_{y \in Y} \frac{|R(y)| (|R(y)| -1)}{2} \right) \\ \vspace{0.3cm}
	& \geq & \frac{1}{|X|} \left( \frac{|X|(|X|-1)}{2} + \sum_{y \in Y} \frac{|X|/|Y| (|X|/|Y| -1)}{2} \right) \\ \vspace{0.3cm}
	& = &  \frac{|X|-1}{2} + \frac{|X|/|Y| -1}{2} \\
	\end{array}
	\]
	
	Let $y \in Y$ be defined such that $N_D^-[v(x)] \cap Y \subseteq N_D^-[y] \cap Y$ and $|N_D^-[y] \cap Y|$ is maximum (such a $y$ exists as $y=v(x)$ is an option).
	Define $Q$ as follows.

\[   Q = \{y\} \cup \left( X \setminus \left(N_D^-[y] \cup N_D^{--}[y] \right) \right) \]

	Note that $Q$ consists of $y\in Y$ and all vertices in $X$ at distance at least $3$ from $y$. As we chose $y$ to be a vertex with maximal in-neighbourhood in $Y$, by Lemma~\ref{cor:T}, $Q$ is a quasi-kernel in $D$. We will show that $Q$ has the desired size.

   Let $u$ be a vertex in $X$ such that $ux \in A(H)$. First consider the case when $v(u)=v(x)$, in which case $u \in R(v(x))$ and
        therefore $u \in N_D^-[v(x)]$, which implies that $u \not\in Q$ as $v(x)=y$ or $v(x)y \in A(D)$.

        So now consider the case when $v(u) \not= v(x)$. In this case $v(u) v(x) \in A(D)$, which implies that $v(u) \in  N_D^-[v(x)] \cap Y \subseteq N_D^-[y] \cap Y$. Therefore $u (v(u)) y$ is a path of length two in $D$, implying that $u \not\in Q$.  So, in all cases when  $ux \in A(H)$ we observe that
$u \not\in Q$.

	By our bound on $d_H^-(x)$ above, we note that the following holds.

	\[
	|Q|  \leq  1 + |X| - \left(  \frac{|X|-1}{2} + \frac{|X|/|Y| -1}{2} +1 \right)
	=  \frac{|X|}{2} - \frac{|X|/|Y|}{2} + 1
	\]
	
	As $n$ is the order of $D$ we note that $n=|X|+|Y|$, which implies the following.
	
	\[
	|Q| \leq  \frac{n-|Y|}{2} - \frac{(n-|Y|)/|Y|}{2} + 1 =  \frac{n}{2} - \frac{|Y|}{2} - \frac{n}{2|Y|} + \frac{1}{2} + 1 = \frac{n+3}{2} - \left( \frac{|Y|}{2} + \frac{n}{2|Y|} \right)
	\]
	
	Let $t>0$ and $f(t) = t/2 + n/(2t).$ Then $f'(t) = 1/2 - n/(2t^2)$, which implies that $f(t)$ takes its minimum when $f'(t)=0$ which in turn is when $t^2 = n$.
	This implies that

	\[
	|Q| \leq \frac{n+3}{2} - \left( \frac{|Y|}{2} + \frac{n}{2|Y|} \right) \leq \frac{n+3}{2} - \left( \frac{\sqrt{n}}{2} + \frac{n}{2\sqrt{n}} \right) = \frac{n+3}{2} - \sqrt{n}
	\]
	
	As $Q$ was a quasi-kernel in $D$ this completes the first part of the proof.

	In order to prove the second part of the theorem we let $k \geq 1$ be any integer and construct the digraph $D_k$ of order $(2k+1)^2$ as follows.
	Let $T$ be a $k$-regular tournament of order $2k+1$ and for each vertex, $v$, of $T$ add $2k$ new vertices, $V_v$, with arcs into $v$. The resulting
	digraph, $D_k$, has order $(2k+1)^2$ and is a one-way split digraph with partition $V(T)$ (the tournament) and $V(D_k) \setminus V(T)$ (the independent set).
	Let $Q$ be a minimum quasi-kernel in $D_k$.
	
	 Since there is no arc from $V(T)$ to $V(D_k)\backslash V(T)$ we may assume that $Q \cap V(T) \not= \emptyset$. As $T$ is a tournament $|Q \cap V(T)|=1$ and let $v$ be the vertex in $Q \cap V(T)$.  As $T$ is a $k$-regular tournament we note that $|N^+(v)|=k$ and $V_x \subseteq Q$ for all $x \in N^+(v)$.  This implies that $|Q| \geq 1 + k(2k) = 2k^2+1$.
	So in all cases $|Q| \geq 2k^2+1$, which implies the following.
	
	\[
	|Q| \geq 2k^2+1 =  \frac{4k^2 + 4k +1}{2} - \frac{4k+2}{2} + \frac{3}{2} = \frac{n}{2} - \sqrt{n} + \frac{3}{2}
	\]
	
	By the first part of the theorem we have $|Q| \leq \frac{n+3}{2} - \sqrt{n}$, which implies that $|Q| = \frac{n+3}{2} - \sqrt{n}$.
\end{proof}

\end{document}